\documentclass[12pt, oneside, reqno]{amsart}
\usepackage{amsmath}
\usepackage{graphicx}
\usepackage{hyperref}
\usepackage[latin1]{inputenc}
\usepackage{mathtools}
\usepackage{amsfonts}
\usepackage{amssymb}
\usepackage[T1]{fontenc}
\usepackage{amsthm}
\usepackage{fullpage}
\usepackage{xcolor}

\newtheorem{theorem}{Theorem}[section]
\newtheorem{corollary}{Corollary}[theorem]
\newtheorem{lemma}[theorem]{Lemma}
\newtheorem{proposition}[theorem]{Proposition}

\theoremstyle{definition}
\newtheorem{definition}{Definition}[section]
 
\theoremstyle{remark}

\numberwithin{equation}{section}

\newcommand{\F}{\mathbb{F}_q}
\newcommand{\Fm}{\mathbb{F}_{q^m}}
\newcommand{\M}{\mathfrak{M}}

\newcommand{\N}{\mathfrak{N}}

\newcommand{\Tt}{T_{(t)}}
\newcommand{\tT}{t_{(T)}}

\title{ Characteristic functions for \MakeLowercase{(r, n)}-free and  \MakeLowercase{(f, g)}-free elements }


\author{Himangshu Hazarika}
\address{Department of Mathematics, Behali Degree College, Borgang, Biswanath, Assam, India }
\email{diku\_95@tezu.ernet.in}

\author{Dhiren Kumar Basnet}
\address{Department of Mathematical Sciences, Tezpur University, Assam, India}
\email{dbasnet@tezu.ernet.in}

\begin{document}

\maketitle

\begin{abstract}

For a prime power $q$, $\F$ denotes the finite field of order $q$, and for $m\geq 2$, $\Fm$ denotes the extension field of degree $m$. We establish a characteristic function for the set of $(r,\, n)$-free elements of finite cyclic $R$-module for the Euclidean domain $R$. Furthermore, we explore $(f,\, g)$-freeness through polynomial values and finally give an expression for the characteristic function for the set of $(f,\, g)$-free elements.  

\end{abstract}

\section{Introduction}\label{sec8.1}

For a prime power $q$, $\F$ denotes the finite field of order $q$ and for $m\geq 2$, $\Fm$ denotes the extended field of order $q^m$. The multiplicative group $\Fm^*$ is a cyclic group whose generators are called \emph{primitive} elements of $\Fm$. For $e|q^m-1$, $\alpha \in \Fm^*$ is called $e$-free element if $d|e$ and $\alpha= \beta^d$, for some $\beta \in \Fm^*$ imply $d=1$. Furthermore $\alpha$ is primitive element if and only if it is $(q^n-1)$-free elements. Further, an element $\alpha \in \Fm$ is called \emph{normal} element of $\Fm$ over $\F$, if its Galois orbit i.e., $\lbrace \alpha, \alpha^q, \ldots, \alpha^{q^{n-1}}\rbrace$ forms a basis of $\Fm$ over $\F$.

 The problem of constructing an algorithm to find a primitive element of a finite field is one of the major open problems in finite field theory. Hence, researchers focus on relatively unchallenging problem, i.e., to create an algorithm to find an element of higher multiplicative order. This problem becomes more useful as the elements of higher multiplicative order may replace primitive elements in various applications. One of the notable work of Gao \cite{Gao} who established algorithm for elements of higher order in $\Fm$ with $m^{log_q m/(4 log_q(2 log_q m))-1/2}$ as the lower bound of the order. Later Popovych estimated lower bounds on elements of higher multiplicative orders of finite fields  in \cite{Popo}. Then Cohen, Kapetanakis and Reis defined $(r,\,n)$-free elements for finite fields and studied various properties of $(r, \, n)$-free elements in \cite{CGL}. They provided the characteristic function for the $(r,\,n)$-free elements, which is an extension of Vinogradov's formula for the characteristic function for primitive elements. From which we generalise the notions of $(r,\,  n)$-free and  $(f,\,  g)$-free elements of the finite cyclic module. Next we establish the characteristic functions and prove some results for $(r,\,  n)$-free and  $(f,\,  g)$-free elements of the finite cyclic module.

In this article, we generalise $(r, n)$-free and  $(f, g)$-free elements of the finite cyclic module. In the Section~\ref{sec8.1.1}, we recall the definition of $n$-primitive elements and its characteristic function. Furthermore, we discuss the definition of $(r,\, n)$-free elements, introduced by Cohen, Kapetanakis and Reis \cite{CGL} along with some basic properties. In  Section~\ref{sec8.2}, we explore the finite cyclic $R$-module. Then we establish some results with additive module in Section~\ref{sec8.3}. In Section~\ref{sec8.4} we define $(f, g)$-free elements and provide some properties. Finally, in Section~\ref{sec8.5}, we define $(f, g)$-freeness of elements through polynomial values and establish some additional results.

\section{Preparation}\label{sec8.1.1}


We begin this section with some definitions.

\begin{definition}
If $n|q^m-1$, then an element $\alpha \in \Fm$ is called \emph{n-primitive} element, if it is of order $(q^m-1)/n$. So primitive elements are nothing but 1-primitive elements.
\end{definition}

The characteristic function of $n$-free elements is given by Carlitz in \cite{LC1}, as following.

\begin{lemma}\label{Re8.1.1}

If $M$ is a divisor of $q^m - 1$, the characteristic function for the set of
elements  with multiplicative order $M$ in $\Fm$ is 
\begin{equation}\label{eq8.1.1}
\Psi_M(\alpha) = \frac{M}{q^m-1}\underset{d|M}{\sum}\frac{\mu(d)}{d}\underset{ord(\kappa)|\frac{d(q^m-1)}{M}}{\sum} \kappa(\alpha),
\end{equation}

where the sum runs over all the characters $\kappa$ such that $\mathrm{ord}(\kappa)$ divides $d(q^m-1)/M$.
\end{lemma}

For  a multiplicative cyclic group, Cohen \emph{et al.} provided the following result in \cite[Lemma~2.5]{CGL}.

\begin{lemma}
For positive integers $r, \, n$,
\begin{equation*}
T(r, n) := \underset{l|r}{\sum} \frac{|\mu(l_{(n)})|}{\phi(l_{(n)})}. \phi(l) = \mathrm{gcd}(r, n) . W (\mathrm{gcd}(r, r_{(n)})),  
\end{equation*}

where $W(k)$ is the number of square-free divisors of $k$ and  $x_{(y)}= \frac{x}{\mathrm{ gcd}(x,\,  y)}.$
\end{lemma}

\subsection*{$(r,\,  n)$-free elements}

Following concepts are provided by Cohen, Kapetanakis and Reis in \cite{CGL}.

\begin{definition}
 Let $C_Q$ denote a multiplicative cyclic group of order $Q$. Let
$n$ be a divisor of $Q$ and $r$ be a divisor of $Q/n$. Then an element $\alpha \in C_Q$ is called $(r,\,  n)$-free if the following hold:
\begin{itemize}
\item[(i)] $\alpha$ is in the subgroup $C_{Q/n}$, i.e., $\mathrm{ord}(\alpha)|\frac{Q}{n}$;

\item[(ii)] $\alpha$ is $r$-free in $C_{Q/n}$, i.e., if $\alpha = \beta^a$ with $\beta \in C_{Q/n}$ and $a|r$, then $a = 1$.
\end{itemize}

\end{definition}

\begin{lemma}\emph{\cite[Lemma~3.3]{CGL}}
Let $n$ be a divisor of $Q$ and $r$ be a divisor of $Q/n$. An element
$\alpha \in C_Q$ is $(r,\,  n)$-free if and only if $\alpha = \beta^n$
for some $\beta \in C_Q$ but $\alpha$ is not of the form $\beta^{np}_0$  where $\beta_0 \in C_Q$, for every prime divisor $p$ of $r$. In particular, $\alpha \in C_Q$ is
$(r,\,  n)$-free if and only if $\mathrm{gcd}\left( rn,\,  \frac{Q}{\mathrm{ord}(\alpha)}\right) = n$.
\end{lemma}

In the next section, we provide some results for finite cyclic $R$-module.

\section{Finite cyclic $R$-module}\label{sec8.2}

Let $R$ be an Euclidean domain and $\M$ be a finite cyclic $R$-module, under the rule $r \, o \,  \alpha$, where $r \in R$ and $\alpha \in \M$. Let $g\in \M$ be a generator of $\M$. Since $\M$ has similar properties as abelian group, hence $\widehat{\M}$ can be defined as an $R$-module, under the rule, $r \, o \, \psi\, : \, \alpha \, \mapsto \, \psi(r \, o \, \alpha)$, for $\alpha \in \M$ and $\psi \in \widehat{\M}$. In this article,  whenever a
sum runs through the divisors of some element of $R$ or when a condition is  applicable to all the members of a conjugacy class of $R/\sim$, we will consider just one representative, as an element of $R$ and will treat as same with its conjugates that hold the equivalence relation $r \sim s  \Leftrightarrow r = us$, for some $u \in R^*$.

  For $\alpha \in \M$,  from the properties of $R$ and $\M$, we have that  the annihilator of $\alpha$ is an ideal of $R$, hence it has a unique generator. This  annihilator is called the \emph{order} of $\alpha$ and we denote it with $\mathrm{ord}(\alpha)$. We set $m := ord(g)$. 

For $d\in R$, the Euler function is defined as $\phi(d)= |(R/d R)^*|$.

For $a,b \in R$, we denote $a_{(b)}$ as $$a_{(b)}= \frac{a}{gcd(a, b)}.$$

\begin{lemma}\label{Re8.2.1} Let $r\in R$ and $b$ is a divisor of $m$. An element $\alpha\in \M$ is of the form $r \circ \beta$ with $ord(\beta)= b$ if and only if $ord(\alpha)= b_{(n)}$, where $n= \mathrm{ gcd}(r,\,  m)$.
\end{lemma}
\begin{proof}
$(\Rightarrow)$ Let $\alpha\in \M$ be of the form $r \circ \beta$ with $ord( \beta )= b$. Then $b$ is the generator of the ideal (ideal of annihilators of $\beta$) in R. Then $\alpha= r\circ \beta$ has order $\frac{b}{\mathrm{gcd }(b,\, r)}= \frac{b}{\mathrm{ gcd} (b, \, n )}= b_{(n)}$.

$(\Leftarrow)$ Let $\mathrm{ ord}(\alpha)= b_{(n)}$, where $n=\mathrm{ gcd}(r,\,  m)$. To show $\alpha\in \M$ is of the form $r \circ \beta$ with $\mathrm{ ord}(\beta)=b$.

Consider the set $A_b= \{ r\circ \beta |\mathrm{ ord}(\beta)=b \}.$ Then the cardinality of the set is $\phi(b_{(n)})$. Now for any element $\delta$ of order $b$, the elements of $A_b$ are $i\circ (r\circ \delta)= (ir)\circ \delta$, where $\mathrm{  gcd}(i,\,  b)=1$. So we have, $(ir)\circ \delta = (jr)\circ \delta$ if and only if $i = j + y b_{(n)}$, for some $y \in R$, i.e., $i\equiv j (\mathrm{ mod} (b_{(n)}))$. (This $\mathrm{ mod}$ is analogous to $natural\,  mod$.)

Therefore the cardinality of the set $A_b$ is the number of in-congruent $x (\mathrm{ mod}(b_{(n)}))$, as $x\in R$ such that $\mathrm{ gcd}(x,\,  b)=1$. Since $b_{(n)}$ divides $b$, hence cardinality is $\phi(b_{(n)})$.

This completes the proof.
\end{proof}

\begin{lemma}\label{Re8.2.2} For $n,\,  r \in R$, we have that,
$$T(r,\, n):= \underset{t|r}{\sum}\frac{|\mu (t_{(n)})|}{\phi (t_{(n)})}. \phi (t)= \left| \left(R/\alpha R\right)\right|. W\left(\left| \left( R/ \beta R\right) \right|\right), $$
where $\alpha =\mathrm{ gcd}(n,\,  r )$, $\beta =\mathrm{ gcd} (r,\,  r_{(n)})$ and $W(a)$ denotes the number of square free divisors of $a$.
\end{lemma}

\begin{proof}

We have $f_n(r):= \frac{|\mu (t_{(n)})|}{\phi (t_{(n)})}.\phi (t)$  and $g_n(r):= \left| \left(R/\alpha R\right)\right|. W\left(\left| \left( R/ \beta R\right) \right|\right) $, which are multiplicative functions in $R$.
 Hence the same holds for $T(n,\, r)$ and hence it suffices to prove the equality $T(n,\, r)= g_n(r)$ in the case, where $r$ is a prime power in $R$. We use $r=p^a$, where $p$ is a prime in $R$ and $a$ is the exponent of $p$ such that $a\in \mathbb{N}$. 
We write $n=p^bn_0$, where  $b$ is an  integer with $b \geq 0$ and $\mathrm{ gcd}(p,\,  n_0)=1$. Now we have the following cases.

\begin{itemize}

\item[(i)] If $b=0$, we have $\mathrm{ gcd}(r,\,  n)=1$, $r_{(n)}=r$ i.e., $\alpha=1,\, \beta = r.$

$T(n,\, r)= \underset{t|r}{\sum}\frac{|\mu (t_{(n)})|}{\phi (t_{(n)})}.\phi(t)= \phi (1) +\frac{|\mu (p)|}{\phi(p)}\phi(p) =2,$

and $g_n(r)=  \left| \left(R/ R\right)\right|. W\left(\left| \left( R/ p^a R\right) \right|\right) =  W\left(\left| \left( R/ p R\right) \right|^a\right)= W\left(\left| \left( R/ p R\right) \right|\right)=2.$


\item[(ii)]  If $0<b \leq a$, then we have $\mathrm{ gcd}(n,\,  r)= p^b$ and $r_{(n)}= p^{a-b}$, i.e. $\alpha= p^b$ and $\beta=  p^{a-b}$ if $b<a$,  $\beta=1$ if $b=a$. 

Then for $b<a$, $T(n,\,  r)= \underset{i=0}{\overset{b}{\sum}}\phi(p^i)+ \underset{i=b+1}{\overset{a}{\sum}}\frac{|\mu (p^{i-b})|}{\phi (p^{i-b})} \phi(p^i)  =  \underset{i=0}{\overset{b}{\sum}}\phi(p^i)+ \frac{1}{\phi (p)} \phi(p^{b+1}) = 2 |(R/p^b R)|.$

If $b=a$, then $T(n,\, r)= |(R/p^b R)|.$

When $b<a$, $\alpha= p^b$ and $\beta= p^{a-b}$ and hence $g_n(r)=  \left| \left(R/ p^b R\right)\right|. W\left(\left| \left( R/ p^{a-b} R\right) \right|\right) =  \left| \left(R/p^b R\right)\right|. W\left(\left| \left( R/ p R\right) \right|^{a-b}\right)=  \left| \left(R/p^b R\right)\right|.W\left(\left| \left( R/ p R\right) \right|\right)=2 \left| \left(R/p^b R\right)\right|. $

Similarly, when $b=a$, then $\alpha= p^b$ and $\beta=1$. In this case
$g_n(r)= |(R/ p^b R)|.$  

\item[(iii)] If $b>a$, then $\mathrm{ gcd}(n,\,  r)= p^a$ and $r_{(n)}=1$, i.e. $\alpha= p^a$, $\beta = 1.$ Then $T(n,\,  r)= \underset{i=0}{\overset{a}{\sum}} \frac{|\mu (1)|}{\phi (1)}\phi (p^i)= |(R/ p^a R)|.$  For $\alpha= p^a$, $\beta=1$, $g_n(r)= |(R/p^a R)|$.

\end{itemize}

\end{proof}


Next we have the characteristic function for the set of elements in $\M$ with order $n$.
\begin{lemma}\emph{\cite[Lemma~3.4]{CGL}} \label{Re8.2.3} If $n$ is a divisor of $m$, then the characteristic function for the set of elements in $\M$ with order $n$ can be expressed as
\begin{equation*}
 \Lambda_n(\omega)= \frac{|(R/ n R)|}{|\M|}\underset{d|n}{\sum}\frac{\mu (n)}{|(R/ d R )|}\underset{ord(\chi)=d}{\sum}\chi(\omega).
\end{equation*}

\end{lemma}

\begin{lemma}\label{Re8.2.4} For $r=m/n$ with $n$ is a divisor of $m$, the above function can be rewritten as 
\begin{equation*}
 \Lambda_n(\omega)= \frac{\phi (m/n)}{|\M|}\underset{t|m}{\sum}\frac{\mu (t_{(n)})}{\phi (t_{(n)})}\underset{ord(\chi)=t}{\sum}\chi(\omega).
\end{equation*}
 
\end{lemma}

\subsection{Generalised $(r,\, n)$-free elements}

Following is the definition of Generalised $(r,\, n)$-free elements.

\begin{definition}  Let $r$ and $n$ be two divisors of $m$. An element $x \in \M$ is called $(r,\, n)$-free if $ ord(a) | \frac{m}{\mathrm{ gcd}(r,\, n)}$ and if $x = d \circ y$ for some $y \in \M$ with $ord(y) | \frac{m}{\mathrm{ gcd}(r,\, n)}$ and $d| r_{(n)}$, implies $d=1$.
\end{definition}

\begin{lemma}\label{Re8.2.5} Let $r,\, n$ be two divisors of $m$. Let $a\in \M$ be such that $ ord(a) | \frac{m}{\mathrm{ gcd}(r,\, n)}$. Then $a$ is $(r,\,n)$-free if and only if $\mathrm{ gcd}(r,\, n)= \mathrm{ gcd} (r,\, \frac{m}{ord(a)})$.
\end{lemma}

\begin{proof}

 Clearly it suffices to prove that $\mathrm{gcd} \left( r_{(n)}, \frac{m}{\mathrm{gcd} (r,\, n) ord(a)}\right)=1$. We take $b\in \M$ of order $\frac{m}{\mathrm{gcd}(r,\, m)}$. Then $a= k \circ b$ for some $k \in R$, while $ord (a) = \left(\frac{m}{\mathrm{gcd}(r,\, n) \mathrm{gcd}\left(k, \frac{m}{\mathrm{gcd}(r,\, n)}\right)} \right).$

($\Rightarrow$) Let $a$ be  $(r,\, n)$-free such that $\mathrm{gcd} \left( r_{(n)}, \frac{m}{\mathrm{gcd} (r,\, n) ord(a)}\right)\neq 1$. 

Then $\mathrm{gcd} \left( r_{(n)}, \mathrm{gcd}\left(k, \frac{m}{\mathrm{gcd}(r,\, n)}\right) \right)\neq 1 $ i.e., $\mathrm{gcd} \left( r_{(n)}, k, \frac{m}{\mathrm{gcd}(r,\, n)} \right)\neq 	1,$ i.e. $k=k_0k_1,$ for some $k_0\in R$ such that $k_0\neq 1$ and $k_0\, | \, \mathrm{gcd} \left( r_{(n)}, \frac{m}{\mathrm{gcd}(r,\, n)}\right).$

Thus $a= k \circ b = (k_0k_1)\circ b =k_0\circ (k_1 \circ b).$ Since $a$ is $(r,\,n)$-free, this implies $k_0=1,$ a contradiction.

($\Leftarrow$)  Let $ \mathrm{gcd} \left( r_{(n)}, \frac{m}{\mathrm{gcd} (r,\, n) ord(a)}\right)=1$. Take some $\beta \in \M$ such that $ord(\beta)| \frac{m}{\mathrm{gcd}(r,\,n)},$ $a=d\circ \beta$ and $d|r_{(n)}.$ It suffices to show that $d=1$.

First we have  $\beta= i\circ b$, for some $i\in R$. It follows that $a= k \circ b= d\circ (i \circ b)= (di)\circ b$. Now $ ord(b)= \frac{m}{\mathrm{gcd}(r,\,n)}$ and $k\circ b= (di)\circ b$, where $b\in \M$ and $k, di \in R.$ Since $R$ is ED,  there exists $t\in R$ such that $k = di + t\frac{m}{\mathrm{gcd}(r,n)}$. Since $d|di$ and $d|\frac{m}{\mathrm{gcd}(r,\,n)}$, hence $d|k$. We have $ \mathrm{gcd} \left( r_{(n)}, \frac{m}{\mathrm{gcd} (r,\, n) ord(a)}\right)=1$ i.e. $\mathrm{gcd}\left( r_{(n)}, k, \frac{m}{\mathrm{gcd}(r,\,n)}\right)=1$ and combining the facts that $d|k$ and $d|\frac{m}{\mathrm{gcd}(r,\, n)}$ we get $d=1$.

\end{proof} 

\begin{lemma}\label{Re8.2.6}
Let $t| m$ and $x\in \M$. Then

\begin{center}
$\underset{ord(\chi)| t}{\sum} \chi(x) =\begin{cases}
                         |(R/t R)|, \,\mbox{ if }\, t|\frac{m}{ord(x)},\\
                         0, \,\mbox{ otherwise}. 
                         \end{cases} $ 
\end{center} 
                         
\end{lemma}
This can be proved from orthogonality relations.
Finally, we define the characteristic function for the set of $(r,\,  n)$-free
elements in $\M$ by
\begin{equation*}
\Lambda_{r, n}(x)= \frac{\phi(r_{(n)})}{|(R/ rR)|} \underset{t|r}{\sum} \frac{\mu(t_{(n)})}{\phi (t_{(n)})} \underset{ord(\chi)=t}{\sum} \chi (x), x\in \M, 
\end{equation*}

as character sum expression.

We prove the following lemma to support our claim.

\begin{lemma}\label{Re8.2.7} Let $x \in \M$, then
\begin{center}
$\Lambda_{r, n}(x)=\begin{cases}
                         1, \,\mbox{ if } x \mbox{ is } (r,\,  n)\mbox{-free},\\
                         0, \,\mbox{ otherwise}. 
                         \end{cases} $ 
\end{center}
\end{lemma} 

\begin{proof}

 Let $r=p_1^{e_1}p_2^{e_2}\ldots p_k^{e_k}$ be a prime factorisation of $r$ in $R$ and exponents $e_i\neq 0$. It follows from the definition that $\omega$ is $(r,\, n)$-free if and only if  it is $(p_i^{e_i},\, n)$-free for each $i=1,2, \ldots , k$. Moreover we have $\Lambda_{e_1e_2,n}(\omega)= \Lambda_{e_1,n}(\omega)\Lambda_{e_2,n}(\omega).$ It follows that, it suffices to restrict our case to $r= p^s$, where $p$ is irreducible over $R$ and $s\neq 0$.

 Consider the case $p^s\nmid n$, then $\mathrm{gcd}(p^s, n)= p^t$, for some $0\leq t < s.$  So $p^s_{(n)}=p^{s-t}$. 

$$\Lambda_{p^s,n}(\omega)= \frac{\phi(p^s_{(n)})}{|(R/p^s R)|}\underset{i=0}{\overset{s}{\sum}} \frac{\mu (p^i_{(n)})}{\phi(p^i_{(n)})}\underset{ord(\psi)=p^i}{\sum}\psi(\omega).$$

Now\begin{align*}
\underset{i=0}{\overset{s}{\sum}} \frac{\mu (p^i_{(n)})}{\phi(p^i_{(n)})}\underset{ord(\psi)=p^i}{\sum}\psi(\omega) =&
 \underset{i=0}{\overset{t}{\sum}} \frac{\mu (p^i_{(n)})}{\phi(p^i_{(n)})}\underset{ord(\psi)=p^i}{\sum}\psi(\omega) +
  \frac{\mu (p^{t+1}_{(n)})}{\phi(p^{t+1}_{(n)})}\underset{ord(\psi)=p^{t+1}}{\sum}\psi(\omega)\\
  =&  \underset{i=0}{\overset{t}{\sum}} \underset{ord(\psi)=p^i}{\sum}\psi(\omega)- \frac{1}{\phi(p)}\underset{ord(\phi)=p^{t+1}}{\sum}\psi(\omega)\\
=& \underset{ord(\psi)|p^t}{\sum}\psi(\omega)- \frac{1}{\phi(p)}\underset{ord(\psi)=p^{t+1}}{\sum}\psi(\omega).
\end{align*} 
Hence
\begin{equation}\label{eq8.2.1}
\Lambda_{p^s,n}(\omega)= \frac{\phi(p^s_{(n)})}{|(R/p^s R)|} \left( \underset{ord(\psi)|p^t}{\sum}\psi(\omega)- \frac{1}{\phi(p)}\underset{ord(\psi)=p^{t+1}}{\sum}\psi(\omega)\right).
\end{equation} 

Now, let $\nu_p(g)$ stands for the exponent of the prime $p$ in the prime factorisation of $g$. Now $\omega$ is $(p^s,\,  n)$-free if and only if $t= \nu_p(m)- \nu_p(ord(\omega))$.

If $\omega$ is $(p^s,\, n)$-free, then the above equality and Lemma~\ref{Re8.2.6} imply that

 $\underset{ord(\chi)|p^t}{\sum}\chi(\omega)=\left| \left( R/p^s R\right) \right|$ and $\underset{ord(\chi)|p^{t+1}}{\sum}\chi(\omega)=0$, which yields 
 $\underset{ord(\chi)=p^{t+1}}{\sum}\chi(\omega)=- \left| \left( R/p^s R\right) \right| $.

 Now, from the Equation~\eqref{eq8.2.1}, we have 
 $\Lambda_{p^s,n}(\omega)=1$.
 
 Next assume that $\omega$ is not $(p^s,\, n)$-free. This means $t > \nu_p(m)- \nu_p(ord(\omega))$ or $t < \nu_p(m)- \nu_p(ord(\omega))$. 
 
 Now, for the case $t> \nu_p(m)- \nu_p(ord(\omega))$, we have $p^t\nmid \frac{m}{ord(\omega)}$. By Lemma~\ref{Re8.2.6}, we have $\underset{ord(\chi)|p^t}{\sum}\chi(\omega)= \underset{ord(\chi)=p^{t+1}}{\sum}\chi(\omega)=0. $ Hence from Equation~\eqref{eq8.2.1}, $\Lambda_{p^s,n}=0.$
 
 Finally if $t < \nu_p(m)- \nu_p(ord(\omega))$, by Lemma~\ref{Re8.2.6},  we have $\underset{ord(\chi)|p^t}{\sum}\chi(\omega)= \left| \left( R/p^t R\right) \right|$ and $\underset{ord(\chi)=p^{t+1}}{\sum}\chi(\omega)=\phi(p^{t+1})= \left| \left( R/p^t R\right) \right| \left(1- \frac{1}{\left| \left( R/p R\right) \right|}\right) $.
 
 Now, we discuss the case $p^s|n$. In this case $\mathrm{gcd}(p^s,n)=p^s$ and $p^s_{(n)}=1$. Therefore
 \begin{equation}
 \Lambda_{p^s,n}(\omega)= \frac{\phi(p^s_{(n)})}{\left| \left( R/p^t R\right) \right|} \underset{i=0}{\overset{s}{\sum}} \frac{\mu(p^i_{(n)})}{\phi(p^i_{(n)})}\underset{ord(\chi)=p^i}{\sum}\chi(\omega)= \frac{1}{\left| \left( R/p^t R\right) \right|} \underset{ord(\chi)|p^s}{\sum} \chi(\omega).
 \end{equation}
In this case $\omega$ is $(p^s,\, n)$-free if and only if $p^s|\frac{m}{ord(\omega)}$. Along with this and from Lemma~\ref{Re8.2.6}, we have that $\Lambda_{p^s,n}(\omega)=1$ if $\omega$ is $(p^s,n)$-free and $\Lambda_{p^s,n}(\omega)=0$, otherwise.

\end{proof}

\section{Additive module}\label{sec8.3}  

The additive group of $\Fm$ can be viewed as an $\F[x]$-module $ \mathfrak{F} \, o\,  x \, := \, \underset{i=0}{\overset{n}{\sum}} F_i x^{q^i}$, where $\mathfrak{F}(X)= \underset{i=0}{\overset{n}{\sum}} F_i X^i  \in \F[X]$. From the   normal basis  theorem, it is clear that the $\F[x]$-module is cyclic, and normal elements or free elements are generators of the module. 

For $\alpha \in \Fm$, the unique  monic polynomial polynomial $g$ of the least degree dividing $x^m-1$ is said to be $\F$-order of $\alpha$ if $g \circ \alpha =0$.
Throughout this section, we denote $Ord(\alpha)$ as the $\F$-order of $\alpha \in \Fm$. For any element $\alpha$ in some extension of $\F$,  $(x^m-1)\circ \alpha= \alpha^{q^m}-\alpha=0$ if and only if $\alpha\in \Fm$. If we set $I_\alpha= \{ g(x)\in \F[x]    \mid  g\circ \alpha=0 \}$, then $I_\alpha$ is an ideal of $\F[x]$ and hence generated by a polynomial, say $h_\alpha(x)$ (we require $h_\alpha(x)$ to be monic).  Then $h_\alpha(x)$ is the $\F$-order of $\alpha$. It is clear that $Ord(\alpha) \in \F[x]$ and $Ord(\alpha) | x^m-1$.


For $f,g\in \F[x]$, we denote $f_{(g)}$ as
 $$f_{(g)}=\frac{f}{\mathrm{gcd}(f,\, g)}.$$

If $f\in \F[x]$, $f=\underset{i=0}{\overset{s}{\sum}}a_ix^i$, we define $L_f(x)= \underset{i=0}{\overset{s}{\sum}}a_ix^{q^i} $ as $q$-associate of $f$. Also for $\alpha\in \Fm$, let $f\circ \alpha= L_f(\alpha)= \underset{i=0}{\overset{s}{\sum}}a_i\alpha^{q^i}$. For $f,g \in \F[x]$, the following hold.
\begin{itemize}
\item[(i)]$L_f(L_g(x))= L_{fg}(x).$
\item[(ii)] $L_f(x) + L_g(x)= L_{f+g}(x).$
\end{itemize}

\begin{lemma}\label{Re8.3.1}
Let $f\in \F[x]$ and $g\in \F[x]$ be a divisor of $x^m-1$. An element $\alpha\in\Fm$ is of the form $f\circ \beta$ with $ Ord(\beta)=g$ if and only if  $ Ord(\alpha)=g_{(h)}$, where $h= \mathrm{ gcd} (f, x^m-1)$.
\end{lemma}

Proof of this lemma is similar to that of Lemma~\ref{Re8.2.1}.

\begin{lemma}\label{Re8.3.2}
Let $f,g \in \Fm[x]$, then we have 
$$T(f,g):= \underset{t|f}{\sum}\frac{|\mu(t_{(g)})|}{\Phi(t_{(g)}) }. \Phi(t)= \left| \frac{\F[x]}{l\F[x]}\right|. W\left( \left| \frac{\F[x]}{h\F[x]}\right|\right) = \left(q^{deg(l)}\right). W\left( q^{deg(h)}\right),$$
where $l=\mathrm{ gcd}(f, g)$ and $h=\mathrm{ gcd}(f, f_{(g)})$, $W(n)$ denotes the number of square free divisors of $n$.
\end{lemma}

Proof of this lemma is similar to that of Lemma~\ref{Re8.2.2}.

\section{ $(f,\,  g)$-free element}\label{sec8.4}
\begin{definition}\label{Def8.4.1}
Let $f, g\in \F[x]$ be divisors of $x^m-1$. An element $\alpha\in\Fm$ is called $(f,\, g)$-free if $ Ord(\alpha) \mid \frac{x^m-1}{\mathrm{gcd}(f,\, g)}$ and $\alpha= h\circ \beta$, for some $\beta\in \Fm$ with $ Ord(\beta)\mid \frac{x^m-1}{\mathrm{ gcd}(f,\,  g)}$ and $h|f_{(g)}$, imply $h=1$.
\end{definition}

Then we have the following.

\begin{lemma}\label{Re8.4.1}
Let $f,\, g \in \F[x]$ be divisors of $x^m-1$. Let $\alpha\in \Fm$ be such that $ Ord(\alpha)\mid \frac{x^m-1}{\mathrm{ gcd} (f,\, g)}$. Then $\alpha$ is $(f,\,  g)$-free if and only if $\mathrm{ gcd}(f,\,  g) = \mathrm{ gcd} (f,\,  \frac{x^m-1}{Ord(\alpha)})$.
\end{lemma}

The proof is similar to Lemma~\ref{Re8.2.5}.

\begin{lemma}\label{Re8.4.2}
Let $t | x^m-1$ and we take some $\omega\in \Fm$. Then 
\begin{center}
$\underset{ Ord(\psi)|t}{\sum}\psi(\omega)=\begin{cases}
                         \left| \frac{\F[x]}{<t>}\right|\, \mbox{or } \, q^{deg(t)} , \,\mbox{ if }\, t\mid \frac{x^m-1}{Ord(\omega)},  \\
                         0, \,\mbox{ otherwise}. 
                         \end{cases}$
\end{center}
\end{lemma}

\begin{definition}\label{Def8.4.2}
Now we define the function 
$$\Omega_{f,g}(\omega)=\frac{\Phi(f_{(g)})}{q^{deg(f)}} \underset{h|f}{\sum}\frac{\mu(h_{(g)})}{\Phi(h_{(g)})}\underset{ Ord(\psi)=h}{\sum}\psi(\omega), \, \omega\in \Fm $$
and prove it to be characteristic function for $(f,\, g)$-free elements of $\Fm$. 
\end{definition}

\begin{lemma}\label{Re8.4.3}

Let $\omega \in \Fm$, then
\begin{center}
$\Omega_{f, g}(\omega)=\begin{cases}
                         1, \,\mbox{ if } \omega \mbox{ is } (f,\,  g)\mbox{-free},\\
                         0, \,\mbox{ otherwise}. 
                         \end{cases} $ 
\end{center}

\end{lemma}

The proof is similar as Lemma~\ref{Re8.2.7}.

\section{ $(f,\,  g)$-freeness through polynomial values}\label{sec8.5}
For divisors $f,\,  F,\,  g,\,  G$ of $x^m-1$ and polynomials $h,\,  H\in \Fm[x]$, we study the number of pairs $(h(y),\,  H(y))$ such that $h(y)$ is $(f,\, g)$-free and $H(y)$ is $(F,\,  G)$-free with $y\in \Fm$. This number can be zero if $h(x)$ and $H(x)$ have certain additive dependence with respect to the polynomials $h,\,  H$.

Following is the definition of dependence for polynomials in $\Fm[x].$

\begin{definition}\label{Def8.5.1} 
Let $f,\,  F\in \Fm[x]$ be divisors of $x^m-1$ and let $h,\,  H\in \Fm[x]$ be non constant polynomials. The pair $(h,\, H)$ is $(f,\, F)$-independent if for every $t|f$ and $T|F$ with $l=\mathrm{ lcm}(t,\,  T)\neq 1$ (we choose $t,\,  T$ to be monic) and $1\leq deg(c) \leq deg(t)$ and $1\leq deg(C) \leq deg(T)$ with $\mathrm{ gcd}(t,\, c)= \mathrm{ gcd}(T,\, C)=1$,  the polynomial
$L_{\frac{lc}{t}}(h(x))+ L_{\frac{lC}{T}}(H(x))$ is not of the form $L_l(g(x))+k$, for some $g\in \F[x]$ and $k\in \Fm$.
\end{definition}

We continue with the following proposition.

\begin{proposition} \label{Prop8.5.1}
With the terms and conditions as above, if $\mathfrak{G}(x)= L_{\frac{lc}{t}}(h(x))+ L_{\frac{lC}{T}}(H(x))$ is not of the form $L_l(g(x))+k$ for any $g(x)\in \F[x]$ and for any $k\in \Fm$, then $\mathfrak{G}(x)$ is definitely not of the form $a(x)^p- a(x)+b$, for any $a(x) \in \F[x]$ and $b\in \Fm$, where $q=p^n$ for some positive integer $n$.
\end{proposition}

\begin{lemma}\label{Re8.5.2}
Let $f,\,  F\in \Fm[x]$ be divisors of $x^m-1$. If the pair $(h,\,  H)$ is not $(f,\,  F)$-independent, then there exist divisors $r,\, R\in \Fm[x] $ of $x^m-1$ in a way 
that there is no element $\theta\in \Fm$ such that $h(\theta)$ is $(f,\, r)$-free and $H(\theta)$ is $(F,\, R)$-free.
\end{lemma}

\begin{proof}

 Since $(h,\,  H)$ is not $(f,\, F$)-independent, there exist some $t|f$ and $T|F$ with $l= \mathrm{ lcm}(t,\, T)\neq 1$ and $1 \leq deg(c) \leq deg(t)$ and $1 \leq deg(C) \leq deg(T)$ with $\mathrm{ gcd}(t,\, c)=\mathrm{ gcd}(T,\, C)=1$, such that for some $g\in\Fm[x]$ and $k\in \Fm[x]$
\begin{center}
$L_{\frac{lc}{t}}(h(x))+ L_{\frac{lC}{T}}(H(x))= L_l(g(x))+k.$
\end{center}
Then for every $\theta\in \Fm$ we have 
\begin{center}
$\left(\frac{lc}{t}\right)\circ h(\theta) + \left( \frac{lC}{T}\right)\circ H(\theta)= l\circ g(\theta) +k. $
\end{center}
Then multiplying by $\frac{x^m-1}{l}$ we have 

\begin{equation}\label{eq8.5.1}
\frac{(x^m-1)c}{t}\circ h(\theta)+ \frac{(x^m-1)C}{T}\circ H(\theta)= \left(\frac{x^m-1}{l}\right)\circ k. 
\end{equation}

First we consider the case $t\neq T$ (we have already chosen $t, T$ as monic). From the Equation~\eqref{eq8.5.1}, we have 
\begin{center}
$\frac{T(x^m-1)c}{t}\circ h(\theta)= \frac{T(x^m-1)}{l}\circ k$ and $\frac{t(x^m-1)C}{T}\circ H(\theta)= \frac{t(x^m-1)}{l}\circ k$
\end{center} 
Since right hand sides are inverses of one another, hence we have 
\begin{align*}
&\frac{Ord (h(\theta))}{\mathrm{ gcd}  \left(Ord(h(\theta)),\frac{x^m-1}{t}.T.c \right)} = 
\frac{ Ord(H(\theta))}{\mathrm{ gcd} \left( Ord(H(\theta)),\frac{x^m-1}{T}.t.C \right)} \\
&\Leftrightarrow t. \mathrm{ gcd}  \left( T, \frac{x^m-1}{Ord(H(\theta))}.t.C\right) = T. \mathrm{ gcd}  \left( t, \frac{x^m-1}{ Ord(h(\theta))}.T.c\right). 
\end{align*}

We divide  both the parts of above by $\mathrm{ gcd} (t,\, T)$ and get
\begin{center}
$t. \mathrm{ gcd}  \left( \Tt, \frac{x^m-1}{Ord (H(\theta))}.\tT.C\right) = T. \mathrm{ gcd}  \left( \tT, \frac{x^m-1}{Ord (h(\theta))}.\Tt.c\right).$
\end{center}

Since $\mathrm{ gcd} (\tT, \Tt.c)= \mathrm{ gcd}  (\Tt, \tT.C)=1$, the above equation becomes 

\begin{equation}\label{eq8.5.2}
t. \mathrm{ gcd}  \left( \Tt, \frac{x^m-1}{Ord  (H(\theta))}\right)=  T. \mathrm{ gcd}  \left( \tT, \frac{x^m-1}{ Ord (h(\theta))}\right).
\end{equation}

If possible, let there be some $\theta\in \Fm$, such that $h(\theta)$ is $(f,r)$-free and $H(\theta)$ is $(F,R)$-free, for some divisors $r, R\in \Fm[x]$ of $x^m-1$. Using Lemma~\ref{Re8.4.1}, we have $\mathrm{ gcd} \left( f, \frac{x^m-1}{Ord (h(\theta))}\right)= \mathrm{ gcd} (f, r)$.  Then combining this with the fact that $\tT|f$, we have that   $\mathrm{ gcd} \left( \tT, \frac{x^m-1}{Ord (h(\theta))}\right)= \mathrm{ gcd} (\tT, r).$ Similarly we have  $\mathrm{ gcd} \left( \Tt, \frac{x^m-1}{ Ord (H(\theta))}\right)= \mathrm{ gcd} (\Tt, R).$  Then Equation~\eqref{eq8.5.2} gives 
\begin{equation}\label{eq8.5.3}
t. \mathrm{ gcd} (\Tt, R)= T. \mathrm{ gcd} (\tT, r).
\end{equation}

Since $t\neq T$, without loss of generality we assume that there is some non constant monic polynomial $d\in\Fm[x]$ such that $d|\tT$. Now we choose $R= \Tt$ and $r= \frac{\tT}{d}$, then Equation~\eqref{eq8.5.3} yields $\frac{tT}{\mathrm{ gcd} (t,T)}=\frac{tT}{d\, \mathrm{ gcd} (t,T)}$, a contradiction.

Now we have the case $t=T$. In this case, Equation~\eqref{eq8.5.1} becomes 
\begin{center}
$\frac{(x^m-1)c}{l}\circ h(\theta) +\frac{(x^m-1)C}{l}\circ H(\theta)= \frac{(x^m-1)}{l}\circ k$.
\end{center} 

Let us set $r=R=l$. In this case if $h(\theta)$ is $(f,\, l)$-free, then its $\F$-order divides $\frac{x^m-1}{l}$, hence $\frac{(x^m-1)c}{l}\circ h(\theta)=0$. Similarly $\frac{(x^m-1)C}{l}\circ H(\theta)=0$ and hence their sum must be zero i.e., $\frac{x^m-1}{l}\circ k= 0$. Then identical arguments lead us to Equation~\eqref{eq8.5.3}, i.e.,  
\begin{center}
$t.\mathrm{ gcd} (T,\,  R)= T. \mathrm{ gcd} (t,\, r).$
\end{center} 
 
Since $\mathrm{ lcm} (t,\, T)$ is non constant monic polynomial in $\Fm[x]$, hence without loss of generality we may assume that there exists some non constant $e\in\Fm[x]$ such that $e|t$. Now we choose $R=T$ and $r= \frac{t}{e}$, then above yields $tT= \frac{tT}{e}$, a contradiction.

This completes the proof.

\end{proof}

\begin{theorem}\label{Re8.5.3}
Let $f,\, F,\, r,\, R \in \Fm[x]$ be divisors of $x^m-1$ and for $h,\, H \in \Fm[x]$, the pair $(h,\, H)$ is $(f,\, F)$-independent. Let $D+1\geq 1$ be the number of distinct roots of $h(x)+H(x)$ in its splitting field over $\F$. The number $\N_{h,\, H}=\N_{h,\, H}(f,F,r,R)$ of elements $\theta\in\Fm$ such that $h(\theta)$ is $(f,\, r)$-free and $H(\theta)$ is $(F,\, R)$-free satisfies 
\begin{center}
$\N_{h,H}=\frac{\Phi(f_{(r)})\Phi(F_{(R)})}{q^{deg(f)}q^{deg(F)}}\left( q+ M(f,r,F,R)\right),$
\end{center}
\end{theorem}
with $|M(f,r,F,R)|\leq D q^{m/2} Q_{f,r}Q_{F,R}$, where $Q_{g,h}:= \left| \frac{\F[x]}{<k>}\right|.W\left(\left| \frac{\F[x]}{<y>}\right| \right)$ with $k= gcd(g,h)$, $y=\mathrm{ gcd}  (g, g_{(h)})$.

\begin{proof}

 By definition we have, $\N_{h,H}=\underset{\omega\in \Fm}{\sum}\Omega_{f,r}(h(\omega))\Omega_{F,R}(H(\omega))$.
From the characteristic function $\Omega$, for \begin{Large}
$A= \frac{\Phi(f_{(r)})\Phi(F_{(R)})}{q^{deg(f)}q^{deg(F)}} $
\end{Large}, we have that 
\begin{align*}
\N_{h,H}/A &= \underset{\omega\in \Fm}{\sum}\left( \underset{t|f}{\sum}\frac{\mu(t_{(r)})}{\Phi(t_{(r)})}\underset{Ord(\psi)}{\sum} \psi(h(\omega))\right) \left( \underset{T|F}{\sum}\frac{\mu(T_{(R)})}{\Phi(T_{(R)})}\underset{Ord(\kappa)}{\sum} \kappa(H(\omega))\right)\\
&= \underset{\underset{T|F}{t|f}}{\sum}\frac{\mu(t_{(r)})\mu(T_{(R)})}{\Phi(t_{(r)})\Phi(T_{(R)})}
\underset{Ord(\kappa)=T, Ord(\psi)=t}{\sum}G_{h,H}
(\psi, \kappa),
\end{align*}

where $G_{h,H}= \underset{\omega\in \Fm}{\sum}\psi(h(\omega))\kappa(H(\omega))$.

Fix $t|f$ and $T|F$ (we choose $t,T$ to be monic), let $\psi, \kappa$ be additive characters of $\F$-orders $t$ and $T$ respectively, and set $l= \mathrm{lcm}(t,\, T)$. Then  
\begin{center}
$\psi(h(\omega))\kappa(H(\omega))= \widetilde\psi \left( \left(\frac{cl}{t}\right) \circ h(\omega)+ \left(\frac{Cl}{T}\right) \circ H(\omega)\right),$ 
\end{center}
for some additive character $\widetilde\psi$ of $\F$-order $l$ and some polynomials $c, C$ such that $1\leq deg(c)\leq deg(t)$ and $1\leq deg(C) \leq deg(T)$ with $\mathrm{gcd}(t,\, c)= \mathrm{gcd}(T,\, C)=1$.

Since the pair $(h,\, H)$ is $(f,\, F)$-independent, the polynomial $\mathfrak{G}(x)=L_{\frac{cl}{t}}(h(x))+ L_{\frac{Cl}{T}}(H(x))$ is of the form $L_l(g(x))+k$ if and only if $l=1$ i.e. $t=T=1$. Therefore from Weil's theorem we have $|G_{h,H}(\psi,\kappa)|\leq D q^{m/2}$ whenever $(t,\, T)=(1,\, 1)$.

For $t=T=1$, we observe that $\psi$ and $\kappa$ are just the trivial additive characters and so $G_{h,H}(\psi, \kappa)=q^{m}- \varepsilon$, where $\varepsilon$ is the number of the roots of $h(x)+ H(x)$ defined over $\Fm$. Since $\varepsilon \leq D+1$, we have 
\begin{center}
$\left| \N_{h,H}/A- q^m \right| \leq D+1 + Dq^{m/2}E,$
\end{center}
where 
\begin{large}
$E= \underset{\underset{(t,T)\neq (1,1)}{t|f, T|F}}{\sum} \frac{\mu(t_{(r)})\mu(T_{(R)})}{\Phi(t_{(r)})\Phi(T_{(R)})}\underset{Ord(\kappa)=T}{\underset{Ord(\psi)=t}{\sum}}   1= T(f,\, r)T(F,\, R)-1$.
\end{large}

Where $T(f,\, g)$ is as in Lemma~\ref{Re8.3.2} and according to the Lemma~\ref{Re8.5.2}, we have the equality $T(f,\, g)= Q_{f,\, g}$ and so 
\begin{center}
$\left| \N_{h,H}/A - q^m\right| \leq D + 1+ D q^{m/2}(Q_{f,r}Q_{F,R}-1)< Dq^{m/2}Q_{f,r}Q_{F,R}$
\end{center}
from where the result follows.

\end{proof}

\begin{corollary}\label{Re8.5.3.1} Let $f,\, F,\, r,\, R,\, h,\, H$ as in the above theorem.
\begin{center}
If $q^{m/2}\geq D Q_{f,r} Q_{F,R}$, then $\N_{h,H}>0.$
\end{center}

\end{corollary}


\begin{thebibliography}{999}


  \bibitem{LC1}
   Carlitz,  L.
   Primitive roots in a finite fields.
   \emph{Transactions of the American Mathematical Society},
   73(3):314-318, 1952. 




   \bibitem{CGL} 
    Cohen,  S. D., Kapetanakis,  G. and Reis,  L.
      The existence of $\F$-primitive points on curves using freeness.
      \emph{ Comptes Rendus. Math\'ematique},
       360, 641-652, 2022.      
      
 
\bibitem{Gao}
       Gao, S.
         Elements of provable high orders in finite fields.
        \emph{Proceedings of the American Mathematical Society},
        127(6), 1615--1623, 1999.


\bibitem{Popo}
         Popovych, R.
             Elements of high order in finite fields of the form $\F [x]/\Phi_r (x)$.
             \emph{Finite Fields and Their Applications},
               18(4), 700--710, 2012.


    \bibitem{LN}
   Lidl, R.  and  Niederreiter, H.
   \emph{Finite Fields}.
   Cambridge University Press,Cambridge,
   2nd edition, 1997.




 \end{thebibliography}
\end{document}